\documentclass[11pt]{amsart}
\usepackage{a4, amsmath, bm, amssymb, mathrsfs}
\usepackage{setspace}
\usepackage{subfiles}
\usepackage{cite}
\usepackage{enumerate}
\usepackage{url}

\usepackage{a4,amsmath,amssymb,amscd,verbatim,bbm,graphicx,enumerate,tikz}

\definecolor{crimson}{rgb}{0.86, 0.08, 0.24}
\definecolor{bleudefrance}{rgb}{0.19, 0.5, 0.91}
\usepackage[pdftex,colorlinks,linkcolor=bleudefrance,citecolor=crimson,filecolor=black]{hyperref}

\usepackage{faktor}
\usepackage{dsfont}
\usepackage{amsthm}

\makeatletter
\newtheorem*{rep@theorem}{\rep@title}
\newcommand{\newreptheorem}[2]{%
\newenvironment{rep#1}[1]{%
 \def\rep@title{#2 \ref{##1}}%
 \begin{rep@theorem}}%
 {\end{rep@theorem}}}
\makeatother

\newtheorem{theorem}{Theorem}
\newtheorem{lemma}[theorem]{Lemma}

\newtheorem{proposition}[theorem]{Proposition}
\newtheorem{corollary}[theorem]{Corollary}

\theoremstyle{definition}

\newenvironment{example}[1][Example]{\begin{trivlist}
\item[\hskip \labelsep {\bfseries #1}]}{\end{trivlist}}

\numberwithin{equation}{section}
\numberwithin{theorem}{section}

\def\N{\mathbb{N}}

\def\Z{\mathbb{Z}}
\def\R{\mathbb{R}}

\begin{document}
  
\normalsize

\bibliographystyle{plain} \title[Distribution in homology classes and discrete fractal dimension]
{Distribution in homology classes and discrete fractal dimension} 

\author{James Everitt} 

\author{Richard Sharp} 
\address{Mathematics Institute, University of Warwick,
Coventry CV4 7AL, U.K.}

\thanks{\copyright 2023. This work is licensed by a CC BY license.}


\begin{abstract}
In this note we examine the proportion of periodic orbits of Anosov flows that lie in an
infinite zero density 
subset of the first homology group. We show that on a logarithmic scale we get convergence
to a discrete fractal dimension.
\end{abstract}

\maketitle

%
%
%
%
%
%
%
%
%

\section{introduction}
There has been a considerable body of research on how closed geodesics on compact negatively curved manifolds and, more generally, periodic orbits of Anosov flows are distributed in homology classes, for example \cite{Anan}, \cite{bab-led}, \cite{Katsuda-Sunada87},
\cite{Katsuda-Sunada90}, \cite{Kotani}, \cite{lalley}, \cite{phillips-sarnak}, \cite{Pollicott-Sharp-GD},
\cite{Sharp93}.
To state these results more precisely, let $\phi^t : M \to M$ be a transitive Anosov flow such that the winding cycle associated to the measure of maximal entropy vanishes. This class of flows 
includes
geodesic flows over compact negatively curved manifolds.
The basic counting result is that the number of of period orbits of length at most $T$ and lying in a homology class $\alpha \in H_1(M,\Z)$ is asymptotic to $(\mathrm{constant}) \times e^{hT}/T^{1+k/2}$, where $h$ is the topological entropy of the flow and $k\ge 0$ is the first Betti number of $M$.
Furthermore, the distribution is Gaussian and the constant above is related to the variance
\cite{lalley}, \cite{pet-ris-crelle}, \cite{Sharp04tams}.

It is also interesting to ask about the distribution of periodic orbits lying in a set $A \subset H_1(M,\Z)$. 
If $A$ is finite, the behaviour follows from that for single homology classes, so we suppose that 
$A$ is infinite. 
This, of course, implies that $H_1(M,\Z)$ is infinite, i.e. $k \ge 1$.
Petridis and Risager \cite{pet-ris-forum} (for compact hyperbolic surfaces)
and Collier and Sharp \cite{collier-sharp} (for Anosov flows for which the measure of maximal entropy has vanishing winding cycle) independently showed if $A$ has positive density then
the proportion of periodic orbits of length at most $T$ lying in $A$ converges to the density of $A$ (with respect to an appropriate norm), as $T \to \infty$.
To state this and our new results more precisely, let $\mathscr P$ denote the set of prime
periodic orbits for $\phi$ and, for $\gamma \in \mathscr P$, let $\ell(\gamma)$ denote the least period of $\gamma$ and $[\gamma] \in H_1(M,\Z)$ denote the homology class of $\gamma$.
It is convenient to ignore any torsion in $H_1(M,\Z)$, so we can think of $H_1(M,\Z)$
as a lattice in $H_1(M,\R)\cong \R^k$. Write $\mathscr P_T = \{\gamma \in \mathscr P \hbox{ : } 
\ell(\gamma) \le T\}$, $\mathscr P_T(\alpha) = \{\gamma \in \mathscr P_T \hbox{ : } [\gamma]=\alpha\}$
and $\mathscr P_T(A) = \bigcup_{\alpha \in A} \mathscr P_T(\alpha)$.
Fixing a norm $\|\cdot\|$ on $H_1(M,\R)$, 
write $\mathfrak N_A(r) = \#\{\alpha \in A \hbox{ : } \|\alpha\|\le r\}$ and
$\mathfrak N(r) = \mathfrak N_{H_1(M,\Z)}(r)$.
We say that $A \subset H_1(M,\Z)$ has density 
$d_{\|\cdot\|}(A)$ (with respect to $\|\cdot\|$) if 
\[
\lim_{r \to \infty} \frac{\mathfrak N_A(r)}{\mathfrak N(r)} = d_{\|\cdot\|}(A).
\]

\begin{proposition}[Collier and Sharp \cite{collier-sharp}, Petridis and Risager \cite{pet-ris-forum}]
Let $\phi^t : M \to M$ be a transitive Anosov flow for which the winding cycle associated to the measure of maximal entropy vanishes.
Then there exists a norm $\|\cdot\|$
on $H_1(M,\R)$ such that if $A \subset H_1(M,\Z)$ has density $d_{\|\cdot\|}(A)$ then
\[
\lim_{T \to \infty} \frac{\#\mathscr P_T(A)}{\#\mathscr P_T} = d_{\|\cdot\|}(A).
\]
\end{proposition}

The norm (defined in Section \ref{sec:anosov} below) is a Euclidean norm determined by the 
second derivative of a pressure function.

Now suppose $A$ has \emph{density zero}. It is interesting to ask whether we can obtain more precise 
information about the behaviour of 
\[
\mathscr D(T,A) := \frac{\#\mathscr P_T(A)}{\#\mathscr P_T}
\]
as $T \to \infty$. If we write $\rho_A(r) = \mathfrak N_A(r)/\mathfrak N(r)$, then the 
naive conjecture is that
$
\mathscr D(T,A)$ is of order $\rho_A(\sqrt{t})$, as $T \to \infty$, and this is consistent with 
case $A =\{\alpha\}$. It is too optimistic to hope that a precise asymptotic relation holds for general $A$. Nevertheless, one might hope for information on the logarithmic scale if we use some notion of discrete fractal dimension. We say that $A$ has {\it discrete mass dimension} $\delta$ if
\[
\frac{\log \mathfrak N_A(r)}{\log r} = \delta
\]
or, equivalently, that if 
\begin{equation}\label{eq:def_of_kappa}
\mathfrak N_A(r) =  r^{\delta}\kappa_A(r)
\end{equation}
then $\lim_{r \to \infty} \log \kappa_A(r)/\log r =0$.
(Note that this is independent of the choice of norm $\|\cdot\|$.)
For a discussion of discrete fractal dimensions, see \cite{BarlowTaylor}.

\begin{example}
Suppose $A \subset \Z$ is given by $A = \{\pm m^2 \hbox{ : } m \in \N\}$ then the discrete mass dimension
of $A$ is $1/2$. More interesting examples appear in percolation theory
(see, for example, \cite{Heydenreich}).
\end{example}

Our main result is the following.

\begin{theorem}\label{thm:main_logarithmic}
Let $\phi^t : M \to M$ be a transitive Anosov flow for which the winding cycle associated to the measure of maximal entropy vanishes.
If $A \subset H_1(M,\Z)$ has discrete mass dimension $\delta$ then
\[
\lim_{T \to \infty} \frac{\log \mathscr D(T,A)}{\log T} = \frac{\delta -k}{2}.
\]
\end{theorem}

Let $\Sigma$ be a compact orientable surface of genus $\mathfrak g \ge 2$ with a Riemannian
metric $g$ of negative curvature and let $T^1\Sigma$ denote the unit tangent bundle. 
Then the natural projection $p : T^1\Sigma \to \Sigma$ induces a homomorphism
$p_\ast : H_1(T^1\Sigma,\Z) \to H_1(\Sigma,\Z) \cong \Z^{2g}$ whose kernal
is the torsion subgroup, and induces a bijection between the prime periodic orbits of the geodesic flow and primitive closed geodesics on $\Sigma$ such that
$\ell(\gamma) = \mathrm{length}_g(p(\gamma))$ and $p_\ast([\gamma]) = [p(\gamma)]$.
If, for $A \subset H_1(\Sigma,\Z)$, we define
$
\mathscr D_{\Sigma}(T,A)$
to be the proportion of closed primitive geodesics on $\Sigma$ with $g$-length at most $T$ and
with homology class in $A$, then we have the following corollary.

\begin{corollary}\label{cor:surfaces}
Let $\Sigma$ be a compact orientable surface of genus $\mathfrak g \ge 2$ with a Riemannian
metric of negative curvature.
If $A \subset H_1(\Sigma,\Z)$ has discrete mass dimension $\delta$ then
\[
\lim_{T \to \infty} \frac{\log \mathscr D_{\Sigma}(T,A)}{\log T} = \frac{\delta -2\mathfrak g}{2}.
\]
\end{corollary}

\section{Anosov flows}\label{sec:anosov}

Let $M$ be a compact Riemannian manifold and $\phi^t : M \to M$ br a transitive
Anosov flow \cite{anosov}, \cite{fisher}. 
We suppose that $M$ has first Betti number $k \ge 1$ and ignore any torsion in 
$H_1(M,\Z)$.
Using the notation of the introduction, we say that
$\phi$ is \emph{homologically full} if the map $\mathscr P \to H_1(M,\Z) : \gamma \mapsto
[\gamma]$ is a surjection. This automatically implies that the flow is weak-mixing 
(since an Anosov flow fails to be weak-mixing only when it is a constant
suspension of an Anosov diffeomorphism \cite{plante}, it which case it can have no
null homologous periodic orbits) and hence that
\[
\#\mathscr P_T \sim \frac{e^{hT}}{hT},
\]
as $T \to \infty$, where $h>0$ is the topological entropy of $\phi$ \cite{margulis1},\cite{PP83}. 
There is a unique measure of maximal entropy $\mu$ for which 
the measure-theoretic entropy 
$h_\mu(\phi) =h$ \cite{Bowen_max}. (See \cite{fisher} for the notions of topological and measure-theoretic entropy for $\phi$.)

Let $\mathscr M_\phi$ denote the set of $\phi$-invariant Borel probability measures on $M$.
For a continuous function $f : M \to \R$, we define its \emph{pressure} $P(f)$ by
\[
P(f) = \sup\left\{h_\nu(\phi) + \int f d\nu \hbox{ : } \nu \in \mathscr M_\phi\right\}.
\]
Given $\nu \in \mathscr M_\phi$, we can define the associated
winding cycle $\Phi_\nu \in H_1(M,\R)$ by
\[
\langle \Phi_\nu, [\omega] \rangle = \int \omega(Z) \, d\nu,
\]
where $[\omega]$ is the cohomology class of the closed $1$-form $\omega$,
$Z$ is the vector generating $\phi$ and $\langle \cdot,\cdot \rangle$ is the duality pairing
(Schwartzmann \cite{Sch}, Verjovsky and Vila Freyer \cite{VV}).
Write $\mathscr B_\phi = \{\Phi_\nu \hbox{ : } \nu \in \mathscr M_\phi\}$;
this is a compact and convex subset of $H_1(M,\R)$.
The assumption that $\phi$ is homologically full 
is equivalent to $0 \in \mathrm{int}(\mathscr{B}_\phi)$ and
implies that there are fully supported measures
$\nu$ for which $\Phi_\nu =0$. We will impose the more stringent condition that $\Phi_\mu=0$, where $\mu$ is the measure of maximal entropy for $\phi$. 
This class includes geodesic flows overs over compact negatively manifolds with
negative sectional curvature. (In the case
considered in Corollary \ref{cor:surfaces}, $\mathscr B_\phi$ may be identified with the
unit-ball for the Federer--Gromov stable norm on $H_1(\Sigma,\R)$ \cite{Fed},\cite{Gromov}.)

Still assuming that $\Phi_\mu=0$, there is an analytic pressure function
$\mathfrak p : H^1(M,\R) \to \R$, defined by
$
\mathfrak p([\omega]) = P(\omega(Z))$ \cite{Katsuda-Sunada90},\cite{Sharp93}. This is a strictly convex function with
positive definite Hessian; it has a unique minimum at $0$.
For $\xi \in H^1(M,\R)$, we define
$\sigma_\xi >0$ by 
\[
\sigma_\xi^{2k} = \det \nabla^2 \mathfrak p(\xi)
\]
and set $\sigma=\sigma_0$.
There is also an analytic entropy function
$\mathfrak h : \mathrm{int}(\mathscr B_\phi) \to \R$ defined by
$
\mathfrak h(\rho) = \sup\{h_\nu(\phi) \hbox{ : } \Phi_\nu =\rho\}$
such that $\mathfrak p$ and $-\mathfrak h$ are Legendre conjugates 
(via the pairing $\langle \cdot,\cdot \rangle$) \cite{Rock}.
More precisely,
$-\nabla \mathfrak h :  \mathrm{int}(\mathscr B_\phi) \to H^1(M,\R)$
and $\nabla \mathfrak p : H^1(M,\R) \to \mathrm{int}(\mathscr B_\phi)$ are inverses
and
\[
\mathfrak h(\rho) = \mathfrak p((\nabla \mathfrak p)^{-1}(\rho)) - \langle (\nabla \mathfrak p)^{-1}(\rho),\rho \rangle.
\]
We write $\xi(\rho) = (\nabla \mathfrak p)^{-1}(\rho)$.
Then
$-\nabla^2 \mathfrak h(\rho) = (\nabla^2 \mathfrak p(\xi(\rho)))^{-1}$.
In particular, $\xi(0)=0$, 
$\mathcal H := -\nabla^2 \mathfrak h(0) = (\nabla^2 \mathfrak p(0))^{-1}$
is positive definite
and 
$\det \mathcal H = (\det \nabla^2 \mathfrak p(0))^{-1}
= \sigma^{-2k}$. 
We use $\mathcal H$ to define a norm $\|\cdot\|$ on $H_1(M,\R)$ by
\[
\|\rho\| = \langle \rho, \mathcal H\rho \rangle.
\]
We note that
\[
\mathfrak N(r) := \{\alpha \in H_1(M,\Z) \hbox{ : } \|\alpha\| \le r\}
\sim \mathfrak v_k \sigma^k r^k,
\]
where $\mathfrak v_k = \pi^{k/2}/\Gamma(k/2+1)$, the volume of the standard unit-ball in $\R^k$.
For small $\rho$, Taylor's theorem gives us the expansion
\begin{equation}\label{eq:taylor}
\mathfrak h(\rho) = h - \|\rho\|^2/2 + O(\|\rho\|^3).
\end{equation}

We now consider the periodic orbits of $\phi^t$. As above, we ignore the torsion in $H_1(M,\Z)$
and treat it as a lattice in $H_1(M,\R)$. We fix a fundamental domain $\mathscr F$ and,
for $\rho \in H_1(M,\R)$, we define $\lfloor \rho \rfloor \in H_1(M,\Z)$ by
$\rho-\lfloor \rho \rfloor \in \mathscr F$.

\begin{proposition}[\cite{bab-led},\cite{lalley},\cite{Sharp93}]\label{prop:uniform_asymptotic}
 Let $\phi^t : M \to M$ be a weak-mixing transitive Anosov flow. If
 $\rho \in \mathrm{int}(\mathscr B_\phi)$ and $\alpha \in H_1(M,\Z)$ then
 \[
 \#\{\gamma \in \mathscr P_T \hbox{ : } [\gamma] = \alpha + \lfloor \rho T \rfloor\}
 \sim
c(\rho)
 e^{\langle \xi(\rho),T\rho - \lfloor T\rho \rfloor -\alpha\rangle}
 \frac{e^{\mathfrak h(\rho)T}}{T^{1+k/2}},
 \]
 as $T \to \infty$,
 uniformly for $\rho$ in any compact subset of $\mathscr B_\phi$,
 where $c(\rho) =1/( (2\pi)^{k/2} \sigma^k_{\xi(\rho)} \mathfrak h(\rho))$.
\end{proposition}

If $\Phi_\mu=0$, we can set $\rho=0$ and recover the asymptotic 
\[
\#\mathscr P_T(\alpha) \sim \frac{1}{(2\pi)^{k/2} \sigma^k h}
 \frac{e^{ hT}}{T^{1+k/2}},
 \]
 originally proved by Katsuda and Sunada \cite{Katsuda-Sunada90}.
 Furthermore, for all sufficiently small $\Delta>0$, we have
 \begin{equation}\label{eqn:uniform_asymp}
 \lim_{T \to \infty} \sup_{\|\alpha| \le \Delta T}
 \left|\frac{T^{1+k/2}\#\mathscr P_T(\alpha)}{c(\alpha/T)e^{\mathfrak h(\alpha/T)T}}
 -1 \right|=0.
 \end{equation}
 
 \section{Proof of Theorem \ref{thm:main_logarithmic}}
 
 \subsection{Upper bound}

In this section we show that $(\delta-k)/2$ gives an upper bound for the limit
in Theorem \ref{thm:main_logarithmic}. The main idea is to use Proposition
\ref{prop:uniform_asymptotic}
and the Taylor expansion of $\mathfrak h(\rho)$ to replace $\mathscr D(T,A)$ with a sum of Gaussian terms over elements of $A$ with norm bounded by $\eta\sqrt{T\log T}$, for
$\eta>0$ chosen sufficiently large that the resulting error decays faster than $T^{(\delta-k)/2}$.

We begin with the trivial observation that
\[
\mathscr D(T,A) - e^{-hT}hT\#\mathscr P_T(A) = o(\mathscr D(T,A)),
\]
so that it is sufficient to consider $e^{-hT}hT\#\mathscr P_T(A)$.
We can make the following approximation.

\begin{lemma}\label{lem:approximate}
For any $\eta>0$,
\begin{align*}
\sum_{\substack{\alpha \in A \\ \|\alpha\| \le \eta\sqrt{T\log T}}}
&\left(\frac{hT\#\mathscr P_T(\alpha)}{e^{hT}}
- \frac{e^{-\|\alpha\|^2/2T}}{(2\pi)^{k/2} \sigma^k T^{k/2}}\right)
\\
&= 
o\left(T^{(\delta-k)/2} (\log T)^{\delta/2} \kappa_A(\eta \sqrt{T \log T})\right),
\end{align*}
where $\kappa_A$ is defined by equation $(\ref{eq:def_of_kappa})$.
\end{lemma}

\begin{proof}
Let $\eta>0$. Clearly (\ref{eqn:uniform_asymp}) still holds if we take the supremum over
$\|\alpha\| \le \eta\sqrt{T\log T}$. Over this set, we have
$c(\alpha/T) = c(0) + O(\|\alpha\|/T) = c(0) + O(\sqrt{\log T}/\sqrt{T})$
and
\[
\mathfrak h\left(\frac{\alpha}{T}\right)T = ht -\frac{\|\alpha\|^2}{2T} + O\left(\frac{\|\alpha\|^3}{T^2}\right)
= hT -\frac{\|\alpha\|^2}{2T}  +O\left(\frac{(\log T)^{3/2}}{\sqrt T}\right).
\]
Substituting these in, we obtain an estimate
\[
\sup_{\|\alpha\| \le \eta \sqrt{T \log T}}
\left|\frac{hT\#\mathscr P_T(\alpha)}{e^{hT}} - 
\frac{e^{-\|\alpha\|^2/2T} e^{q(\alpha,T)}}{(2\pi)^{k/2} \sigma^k T^{k/2}}\right| = o(T^{-k/2}),
\]
where $|q(\alpha, T)| \le c'(\log T)^{3/2} T^{-1/2}$, for some $c'>0$.
A simple calculation then shows that we may remove the $q(\alpha,T)$ terms, while keeping the
$o(T^{-k/2})$ error term.
To complete the proof, we note that summing over $\|\alpha\|\le \eta \sqrt{T\log T}$
involves $\mathfrak N_A(\eta \sqrt{T \log T}) = O(T^{\delta/2} (\log T)^{\delta/2} 
\kappa_A(\eta \sqrt{T\log T}))$ summands.
\end{proof}

Next we estimate the Gaussian part from the previous lemma.

\begin{lemma}\label{lem:estimate_gaussian}
\[
\sum_{\substack{\alpha \in A \\ \|\alpha\|\le \eta\sqrt{T\log T}}}
\frac{e^{-\|\alpha\|^2/2T}}{(2\pi)^{k/2} \sigma^k T^{k/2}}
=O(T^{(\delta-k)/2} (\log T)^{\delta/2} \kappa_A(\eta\sqrt{T \log T})).
\]
\end{lemma}

\begin{proof}
The result follows from the elementary estimate
\[
\sum_{\substack{\alpha \in A \\ \|\alpha\|\le \eta\sqrt{T\log T}}}
\frac{e^{-\|\alpha\|^2/2T}}{(2\pi)^{k/2} \sigma^k T^{k/2}}
=O\left(\frac{\mathfrak N_A(\eta\sqrt{T \log T})}{T^{k/2}}\right).
\]
\end{proof}

The contribution from $\|\alpha\|>\eta\sqrt{T\log T}$ is estimated as follows.

\begin{lemma}\label{lem:large_terms}
\[
\sum_{\substack{\alpha \in A \\ \|\alpha\|>\eta\sqrt{T\log T}}}
\frac{hT\#\mathscr P_T(\alpha)}{e^{hT}} 
=O(T^{-\eta^2/2} (\log T)^{3k/2 -2}).
\]
\end{lemma}

\begin{proof}
Applying Proposition \ref{prop:uniform_asymptotic}, we see that, for $x \in \R^k$
and $\mathcal C(\Delta)$ a cube of (small) side length $\Delta$ based at $0$,
\begin{align*}
&hTe^{-hT} \#\left\{\gamma \in \mathscr P_T \hbox{ : }
\frac{[\gamma]}{ \sqrt{T\log T}} \in x+\mathcal C(\Delta)\right\}
\\
&\sim
he^{-hT} c(x \sqrt{(\log T)/T})
\frac{e^{\mathfrak h(x \sqrt{(\log T)/T})T}}{T^{k/2}} (\Delta \sqrt{T \log T})^k
\\
&\sim \frac{\Delta^k (\log T)^{k/2} e^{-(\|x\|^2  \log T)/2}}{(2\pi)^{k/2} \sigma^k}.
\end{align*}
Thus we can estimate the sum in the statement by $(\log T)^{k/2}I_\eta(T)$,
where $I_\eta(T)$ is the integral
\[
I_\eta(T) :=\frac{1}{(2\pi)^{k/2}\sigma^k} \int_{B(\eta)} e^{-\|x\|^2  \log T/2} \, dx,
\]
where $B(\eta) = \{x \in \R^k \hbox{ : } \|x\|>\eta\}$.
Substituting $u = x\sqrt{\log T}$ and passing to coordinates $(r,\theta)$ with $r>0$ and
$\|\theta\|=1$, we obtain
\[
I_\eta(T) = \frac{\mathrm{Area}(\{\theta \hbox{ : } \|\theta\|=1\}}{(2\pi)^{k/2}\sigma^k} 
\int_{\eta\sqrt{T}}^\infty e^{-r^2/2} r^{k-1} \, dr
=O(T^{-\eta^2/2} (\log T)^{k-2}),
\]
where we have used standard asymptotics for the complementary error function $\mathrm{erfc}(z)$.
\end{proof}

To complete the proof of the upper bound, choose $\eta >\sqrt{k-\delta}$. 
Then combining Lemmas \ref{lem:approximate}, \ref{lem:estimate_gaussian} and
\ref{lem:large_terms} and noting that
\[
\lim_{T \to \infty} \frac{\log \kappa_A(\eta \sqrt{T \log T})}{\log T} =
\lim_{T \to \infty} \frac{\log \kappa_A(\eta\sqrt{T\log T})}{\log(\eta \sqrt{T\log T})}
\, \frac{\log (\eta \sqrt{T\log T})}{\log T} =0
\]
shows that
\begin{equation}\label{eqn:upper_bound}
\limsup_{T \to \infty} \frac{\log \mathscr D(T,A)}{\log T} \le \frac{\delta-k}{2}.
\end{equation}

\subsection{Lower bound}

Since we seek a lower bound, we only need to consider
\[
\sum_{\substack{\alpha \in A \\ \|\alpha\| \le \sqrt{T}}}
\frac{hT\#\mathscr P_T(\alpha)}{e^{hT}}.
\]
The following result is almost identical to Lemma \ref{lem:approximate} and we do not 
repeat the proof.

\begin{lemma}\label{lem:approximate_version2}
\begin{align*}
\sum_{\substack{\alpha \in A \\ \|\alpha\| \le \sqrt{T}}}
\left(\frac{hT\#\mathscr P_T(\alpha)}{e^{hT}}
- \frac{e^{-\|\alpha\|^2/2T}}{(2\pi)^{k/2} \sigma^k T^{k/2}}\right)
= 
o\left(T^{(\delta-k)/2} \kappa_A( \sqrt{T})\right).
\end{align*}
\end{lemma}

Since we have the bound
\begin{align*}
\sum_{\substack{\alpha \in A \\ \|\alpha\| \le \sqrt{T}}}
 \frac{e^{-\|\alpha\|^2/2T}}{(2\pi)^{k/2} \sigma^k T^{k/2}}
 \ge \frac{e^{-2}}{(2\pi)^{k/2} \sigma^k T^{k/2}} \mathfrak N_A(\sqrt{T})
 =  \frac{e^{-2}}{(2\pi)^{k/2} \sigma^k} T^{(\delta-k)/2} \kappa_A(\sqrt{T}),
\end{align*}
we conclude that 
\[
\liminf_{T \to \infty} \frac{\log \mathscr D(T,A)}{\log T} \ge \frac{\delta-k}{2}.
\]

\end{document}